\documentclass[12pt, oneside]{amsart}
\usepackage[utf8]{inputenc}
\usepackage{amsmath}
\usepackage{amsfonts}
\usepackage{amssymb}
\usepackage{amsthm}
\usepackage{tikz-cd}
\usepackage{graphicx}
\usepackage{geometry}
\theoremstyle{remark}

\newtheorem{theorem}{Theorem}[section]
\newtheorem{corollary}[theorem]{Corollary}
\newtheorem{lemma}[theorem]{Lemma}
\newtheorem{conjecture}[theorem]{Conjecture}
\theoremstyle{definition}

\title[On the homology of $B\Gamma_n^\mathbb{C}$]{On the homology of $B\Gamma_n^\mathbb{C}$ and its application to complex structures on open manifolds}
\author{Filip Samuelsen}
\address{Department of Mathematics, State University of New York, Stony Brook, New York 11790}
\begin{document}
\begin{abstract}
Since the 1970s, it has been known that any open connected manifold 
of dimension 2, 4 or 6 admits a complex analytic structure 
whenever its tangent bundle admits a complex linear structure. 
For half a century, this has been conjectured to hold true for manifolds 
of any dimension. 
In this paper, we extend the result to manifolds of dimension 8. 

To prove the result new $\Gamma_n^\mathbb{C}$-structures on $\mathbb{CP}^n$ are constructed. 
As a consequence we derive a theorem concerning the homology 
of Haefligers classifying space, $B\Gamma_n^\mathbb{C}$. The result then follows
from obstruction theory. 
\end{abstract}
\maketitle
\section{Introduction} \noindent
An almost complex structure on a smooth manifold $M$ of dimension $2n$ is a complex linear structure on 
its tangent bundle $TM$, that is a vector bundle map $J:TM \to TM$ covering the identity 
and such that $J \circ J = -\mathrm{Id}$. This is equivalent to a reduction of the structure group of 
$TM$ to $\mathrm{GL}(n, \mathbb{C})$. A complex analytic structure on $M$ gives rise to an underlying 
almost complex structure. The almost complex structures which arise in this way are said to be integrable. 

A basic problem is to determine when a given almost complex structure 
is homotopic to an integrable one. 
In order to make any progress at all, we make one crucial concession --  
we restrict ourselves to considering the problem for open manifolds. 
That is manifolds without any compact components. 

At the time of writing, there is no known example of an open manifold admitting an almost 
complex structure, but not admitting any complex analytic coordinates. 
Thus, the following has been conjectured  
\begin{conjecture} \label{AlmostComplexImpliesComplex}
Let $M$ be a manifold with no compact components. If $J: TM \to TM$ is an almost complex structure 
on $M$ then $J$ is homotopic to an integrable almost complex structure. 
\end{conjecture}
The analogous conjecture for symplectic manifolds was proven by Gromov using his h-principle 
for open manifolds \cite{Eliashberg}. That is, if $M$ is an open almost complex manifold, then $M$ admits 
a symplectic structure, i.e. a non-degenerate closed two form $\omega$. 
Moreover, one can freely choose the cohomology class of $\omega$ within $H^2(M, \mathbb{R})$. 
In this paper, we make the following partial progress towards conjecture \ref{AlmostComplexImpliesComplex}
\begin{theorem} \label{Dimension10AndLess}
Let $M$ be an open manifold of dimension at most $8$. 
Then $M$ admits a complex analytic structure if and only if it admits an 
almost complex structure. Moreover, any almost complex structure $J$ 
is homotopic to an integrable one. 
\end{theorem}
We now explain the main ingredients that go into the proof. 
Let $\Gamma_n^{\mathbb{C}}$ denote the topological groupoid of germs of local biholomorphisms of
$\mathbb{C}^n$ equipped with the sheaf topology and let $B\Gamma_n^{\mathbb{C}}$ be a CW-complex 
representing the classifying space for $\Gamma_n^{\mathbb{C}}$-structures. 
Differentiating a germ at its source induces a continuous homomorphism 
$\Gamma_n^{\mathbb{C}} \to \mathrm{GL}(n, \mathbb{C})$, hence also a continuous map 
$$Bd: B\Gamma_n^\mathbb{C} \to B\mathrm{GL}(n, \mathbb{C})$$
We convert this map to a fibration and denote the homotopy fiber by 
$B\overline{\Gamma_n^\mathbb{C}}$. 
Let $M$ be an open manifold and $J$ be an almost complex structure on 
$TM$. Then, there is a map $f:M \to B \mathrm{GL}(n, \mathbb{C})$ classifying $(TM, J)$. 
A special case of Haefligers theory \cite{Haefligger1, Haefligger2} asserts that
there exists a lift of $f$ to $B\Gamma_n^\mathbb{C}$ if and only if 
$J$ is homotopic to an integrable almost complex structure.
\begin{center}
\begin{tikzcd}
 & B\Gamma_n^{\mathbb{C}} \arrow[d, "Bd"] & \arrow[l] B\overline{\Gamma_n^\mathbb{C}} \\ 
M \arrow[ur, dashed, "\mathcal{F}"] \arrow[r, "f"] & B\mathrm{GL}(n, \mathbb{C})&
\end{tikzcd}
\end{center}
So for the basic problem mentioned initially, it is important to understand 
the topology of $B\overline{\Gamma_n^\mathbb{C}}$. 
Here is an example of what we would like:
\begin{conjecture}[The Holomorphic Haefliger-Thurston Conjecture] \label{HolomorphicHaefligerThurstonConjecture}
$\pi_*(B\overline{\Gamma}_n^\mathbb{C}) = 0$ for $* \leq 2n$.
\end{conjecture}
This is a holomorphic analogue of the Haefliger-Thurston conjecture \cite{Nariman}.
The main reason for making conjecture \ref{HolomorphicHaefligerThurstonConjecture} is that it is known that 
$\pi_{2n+1}(B\overline{\Gamma}_n^\mathbb{C})$ is infinitely generated \cite{Bott}. 
At the time of writing the following is known
\begin{theorem}[The Connectivity Theorem]
$\pi_*(B\overline{\Gamma_n^\mathbb{C}}) = 0$ for $* \leq n$.
\end{theorem} \noindent
The case $*  < n$ is due to Landweber \cite{Landweber} and the case 
$*=n$ is due to Adachi \cite{Adachi1}. For $n=1$ conjecture 
\ref{HolomorphicHaefligerThurstonConjecture} has been proven by Haefliger–Sithanantham \cite{Haefliger-Sithanantham}.
Besides the theorem of Haefliger–Sithanantham, there is at the time of writing no information about the homotopy groups 
of $B\overline{\Gamma_n^\mathbb{C}}$ in the range between $n$ and $2n$.

Combining the connectivity theorem with the fact that $B\mathrm{GL}(n, \mathbb{C})$ admits a CW-composition consisting only of cells in even dimension, one has that any almost complex structure on an open manifold of 
dimension at most 6 is homotopic to an integrable one. (See \cite{Adachi2}).
In order the extend this result to manifolds of dimension 8, we will prove a theorem about the integral 
homology of $B\Gamma_n^\mathbb{C}$. 
\begin{theorem}[The Homology Surjection Theorem] \label{HomologySurjectivityTheorem}
The map $Bd_*: H_*(B\Gamma_n^\mathbb{C}, \mathbb{Z}) \to H_*(B\mathrm{GL}(n, \mathbb{C}), \mathbb{Z})$ 
is surjective for $* \leq 2n$.
\end{theorem}
This theorem is an improvement on a theorem of Hurder \cite{Hurder}, who 
proved the same statement for homology with rational coefficients. Our proof is similar to that of 
Hurder, which is based on the rational independence of Chern numbers. The refinement to integral coefficient requires a construction of new $\Gamma_n^\mathbb{C}$-structures on $\mathbb{CP}^n$, 
which does not previously appear anywhere in the literature on $\Gamma_n^\mathbb{C}$-structures.

The following corollary of theorem \ref{HomologySurjectivityTheorem} follows immediately from the universal coefficient theorem together with 
the fact that $H_*(B\mathrm{GL}(n, \mathbb{C}), \mathbb{Z}) = 0$ for $\mathbb{*}$ odd. 
\begin{corollary} \label{InjectivityCorollary}
For any abelian group $G$ (not necessarily finitely generated), the map 
$Bd^*: H^*(B\mathrm{GL}(n, \mathbb{C}), G) \to H^*(B\Gamma_n^\mathbb{C}, G)$ is injective for $* \leq 2n$.
\end{corollary}
By a theorem of Bott\cite{Bott}, $Bd^*: H^*(B\mathrm{GL}(n, \mathbb{C}), \mathbb{Q}) \to 
H^*(B\Gamma_n^\mathbb{C}, \mathbb{Q})$ is the zero map for $*>2n$ so this result is sharp.
We end the introduction by deducing theorem \ref{Dimension10AndLess} from theorem 
\ref{HomologySurjectivityTheorem}.
\begin{proof}[Proof of theorem \ref{Dimension10AndLess}]
	Let $M$ be an open almost complex $8$-manifold, and let $f:M \to B\mathrm{GL}(4, \mathbb{C})$ be a map classifying its complex tangent bundle. 
	By Haefligers theory \cite{Haefligger1, Haefligger2}, it suffices to construct a lift of $f$ to $B\Gamma_4^\mathbb{C}$.
	Since $M$ is open $M$ is homotopic to its $7$-skeleton, and thus 
	combining the cellular approximation theorem with the fact that $B\mathrm{GL}(4, \mathbb{C})$ only has even cells, the map 
	$f:M \to B\mathrm{GL}(4, \mathbb{C})$ is homotopic to a map that factors through the $6$-skeleton of 
	$B\mathrm{GL}(4, \mathbb{C})$. 
	Thus, it suffices to construct a section of $Bd: B\Gamma_4^\mathbb{C} \to B\mathrm{GL}(4, \mathbb{C})$ over 
	the $6$-skeleton of $B\mathrm{GL}(4, \mathbb{C})$.
	
	By the connectivity theorem $\pi_*(B\overline{\Gamma_4^\mathbb{C}})) = 0$ for $* < 5$, 
	so the primary (and only) obstruction for lifting the $6$-skeleton of $B\mathrm{GL}(4, \mathbb{C})$ 
	is a cohomology class in $H^6(B\mathrm{GL}(4, \mathbb{C}), \pi_5(B\overline{\Gamma_4^\mathbb{C}})))$.
	Naturality of the primary obstruction implies that the obstruction lies in the kernel of 
	$Bd^*: H^6(B\mathrm{GL}(4, \mathbb{C}), \pi_5(B\overline{\Gamma_4^\mathbb{C}})) 
	\to H^6(B\Gamma_4^\mathbb{C}, \pi_5(B\overline{\Gamma_4^\mathbb{C}}))$. Thus,  
	corollary \ref{InjectivityCorollary} implies that there exists a section over the $6$-skeleton of 
	$B\mathrm{GL}(4, \mathbb{C})$. 
\end{proof}

\textbf{Acknowledgments:}
This article is largely based on the author’s doctoral dissertation, written under the supervision of Dennis Sullivan. The author is grateful to Dennis Sullivan, H. Blaine Lawson, Marie-Louise Michelsohn, and Scott O. Wilson for their valuable comments on this work. He also gratefully acknowledges the generous support provided by the Simons Foundation International during his doctoral studies.

\section{Examples of $\Gamma_n^{\mathbb{C}}$-structures} \label{Examples} \noindent 
In this section, we produce a wide variety of examples of $\Gamma_n^{\mathbb{C}}$-structures. 
They all arise from holomorphic foliations on open complex manifolds. The setup is as follows. 
First, consider a compact complex manifold $M$ of complex dimension $n$. Let $\underline{\mathbb{C}}$ 
denote the trivial rank $1$ complex vector bundle over $M$, 
and suppose that $TM \oplus \underline{\mathbb{C}}$ admits a holomorphic line subbundle $L$. 
Then, pulling back along the natural (holomorphic) projection 
$\mathrm{pr}_M: M \times \mathbb{C} \to M$, we obtain a holomorphic line distribution 
on $M \times \mathbb{C}$. Since the distribution has complex dimension 1, it is automatically integrable, 
and thus determines a holomorphic foliation of complex codimension $n$, 
which we view as a $\Gamma_n^{\mathbb{C}}$-structures on $M \times \mathbb{C}$. Up to homotopy, we may 
equate it with a $\Gamma_n^{\mathbb{C}}$-structure on $M$.

We begin by deducing a sufficient condition for a holomorphic vector bundle to admit a nowhere vanishing global holomorphic section. Recall that a holomorphic vector bundle $E \to M$ of rank $k$ is said to be globally generated if for every 
$p \in M$ there exist $k$ global holomorphic sections $\sigma_1, \ldots \sigma_k: M \to E$ such that the vectors 
$\sigma_1(p), \ldots, \sigma_n(p)$ are linearly independent. 
\begin{lemma}
Let $M$ be a compact complex manifold of dimension $n$ and $E$ a holomorphic vector bundle 
over $M$ of rank $r>n$. If $E$ is globally generated, then there exists a holomorphic 
section which vanishes nowhere.
\end{lemma}
\begin{proof}
Let $V$ denote the vector space of global holomorphic sections. Since $M$ is compact, $V$ is finite 
dimensional; say $\mathrm{dim}(V) = N$. Consider the set of marked sections vanishing 
at their marked point  
$$
M \times V \supseteq Z = \{(x,s) | s(x) = 0\}
$$
There are natural projections $Z \xrightarrow{p} M$ and $Z \xrightarrow{q} V$. 
Since $E$ is globally generated, then for each $x \in M$, the fiber $p^{-1}(x) \subseteq V$ is a linear subspace of codimension $r$, hence dimension $N-r$. Thus 
$$\mathrm{dim}(Z) = \mathrm{dim}(M) + \mathrm{dim}(p^{-1}(x)) = n + N - r < N = \mathrm{dim}(V)$$
because $r > n$. So the projection $Z \xrightarrow{q} V$ cannot be onto. However, if every global section vanished somewhere, then $q(Z) = V$, a contradiction. Hence, there exists a global section that vanishes nowhere.
\end{proof}
Next we show how a holomorphic line bundle of $TM \oplus \mathbb{C}$ give rise to a corresponding holomorphic line distribution on $M \times \mathbb{C}$. 
\begin{lemma} \label{SecondLemma}
	Let $M$ be a compact complex manifold, set $X = M \times \mathbb{C}$ and let $\pi: X \to M$ be the natural projection. 
	If $TM \oplus \underline{\mathbb{C}}$ admits a holomorphic line subbundle isomorphic to $L$, 
	then $TX$ admits a holomorphic line bundle isomorphic to $\pi^*(L)$.
\end{lemma}
\begin{proof}
	Clearly, $TM \oplus \underline{\mathbb{C}}$ admits a holomorphic line bundle isomorphic to $L$ if and only if 
	$\mathrm{Hom}(L, TM \oplus \mathbb{C})$ admits a nowhere vanishing holomorphic section. 
	Since $TX = TM \times T\mathbb{C} = \pi^*(TM \oplus \mathbb{C})$, then 
	$$\mathrm{Hom}(\pi^*(L), TX) = \mathrm{Hom}(\pi^*(L), \pi^*(TM \oplus \underline{\mathbb{C}}))= \pi^*(\mathrm{Hom}(L, TM \oplus \mathbb{C}))$$ 
	by the universal property of the pullback. 
	Thus, a nowhere vanishing holomorphic section of $\mathrm{Hom}(L, TM \oplus \underline{\mathbb{C}})$ pulls back to 
	a nowhere vanishing holomorphic section of $\mathrm{Hom}(\pi^*(L), TX)$. 
	But, $TX$ admits a holomorphic line bundle isomorphic to $\pi^*(L)$ if and only if 
	$\mathrm{Hom}(\pi^*(L), TX)$ admits a nowhere vanishing holomorphic section.
\end{proof}
\noindent
Let $E$ be a $n$-dimensional complex vector bundle over a closed oriented $2n$-dimensional manifold $M$. 
The splitting principle implies that, there is a 
manifold $\tilde{M}$ and a smooth map $f: \tilde{M} \to M$ inducing an injection on the cohomology groups 
$$
f^*: H^k(M, \mathbb{Z}) \to H^k(\tilde{M}, \mathbb{Z} )
$$ and such that the pullback bundle $f^*(E)$ is a 
direct sum of line bundles. Therefore, the total Chern class $c(f^*(E))$ can be identified with 
$$
(1 + t_1) \cdots (1 + t_n) \in H^*(\tilde{M}, \mathbb{Z})
$$
for some $t_i \in H^2(\tilde{M}, \mathbb{Z})$. 
The classes $t_1, \ldots, t_k$ are called the Chern 
roots of $E$. The point being, that the symmetric functions on $t_1, \ldots, t_k$ can always be 
identified with well defined cohomology classes on 
$M$.

Define two monomials in 
$t_1, \ldots, t_k$ to be equivalent if some permutation of $t_1, \ldots, t_k$ 
transforms one into the other. 
For any partition $I=i_1, \ldots, i_r$ of $n$, we denote by 
$s_I(E)=s_{i_1, \ldots, i_r}(E)$ 
the summation of all monomials in the Chern roots which are equivalent to 
$t_1^{i_1} \cdots t_r^{i_r}$. This is a symmetric polynomial and thus determines a well defined cohomology class on $M$. Likewise, we denote by $s_I(E)[M]$ 
the integer obtained by evaluating $s_I(E)$ on the fundamental class. For further information on this approach to the characteristic classes of complex vector bundles, see \cite{MilSta}.

\begin{theorem} \label{ExampleTheorem}
	For every $n > 0$, there exist a closed oriented manifold $M^{2n}$ (not necessarily connected) and a $\Gamma_n^{\mathbb{C}}$-structure $\mathcal{H}$ on $M$ such that $s_n(N\mathcal{H})[M] = 1$. 
\end{theorem}
\begin{proof}
	Both $T\mathbb{CP}^n$ and $\mathcal{O}(1) \to \mathbb{CP}^n$ are globally generated, so the bundle  $$\mathrm{Hom}(\mathcal{O}(-1), T\mathbb{CP}^n \oplus \underline{\mathbb{C}}) = \mathcal{O}(1) \otimes T\mathbb{CP}^n \oplus \mathcal{O}(1)$$ 
	is globally generated. By the first lemma, it therefore admits a nowhere vanishing holomorphic section.  
	By lemma \ref{SecondLemma}, the manifold $X = \mathbb{CP}^n \times \mathbb{C}$ admits a holomorphic subbundle $L \subseteq TX$ isomorphic 
	to $\pi^*(\mathcal{O}(-1))$. Since $\mathcal{O}(-1)$ is a line bundle, $L$ is involutive, and thus by Frobenius' theorem, $L$ determines a holomorphic foliation $\mathcal{F}$ on $X$. 
	The normal bundle of $\mathcal{F}$ is given by 
	$$N\mathcal{F} = T\mathbb{CP}^n \times T\mathbb{C} / \pi^*(\mathcal{O}(-1))$$ 
	Since $X$ is homotopic to $\mathbb{CP}^n$, then $\mathcal{F}$ determines a $\Gamma_{n}^{\mathbb{C}}$-structure $\mathcal{F}|_{\mathbb{CP}^n}$ on $\mathbb{CP}^n$ 
	with normal bundle $N\mathcal{F}|_{\mathbb{CP}^n}$. We then have 
	$$s_n(N\mathcal{F}|_{\mathbb{CP}^n})[\mathbb{CP}^n] = s_n(T\mathbb{CP}^n)[\mathbb{CP}^n] - \mathcal{O}(-1)[\mathbb{CP}^n] = (n+1) - (-1)^n$$ 
	Finally, take $M = \mathbb{CP}^n \sqcup \overline{\mathbb{CP}^n}$, and define the $\Gamma_n^{\mathbb{C}}$-structure $\mathcal{H}$ on $M$ to be induced by $\mathcal{F}|_{\mathbb{CP}^n}$ on one component and by the standard complex structure on the other. Then 
	$$
	s_n(N\mathcal{H})[M] = 	s_n(N\mathcal{F}|_{\mathbb{CP}^n})[\mathbb{CP}^n] + s_n(T\mathbb{CP}^n)[\overline{\mathbb{CP}^n}] = -(-1)^n = 1
	$$
	for $n$ odd. Interchanging the structures we get 
	$$
	s_n(N\mathcal{H})[M] = 	s_n(T\mathbb{CP}^n)[\mathbb{CP}^n] + s_n(N\mathcal{F}|_{\mathbb{CP}^n})[\overline{\mathbb{CP}^n}] = (-1)^n = 1
	$$
	for $n$ even.
\end{proof}

\section{Proof of the homology surjection theorem} \label{ProofHomology} \noindent
In this section, we show how to deduce the 
Homology Surjection Theorem from theorem \ref{ExampleTheorem}.
First, we need the following standard result about independence of Chern numbers.
\begin{theorem} \label{LinearIndependence}
For every $k \leq n$ let $\mathcal{F}^k$ be a $\Gamma_k^{\mathbb{C}}$-structure on a closed oriented manifold 
$M^{2k}$ of dimension $2k$ with $s_k(N\mathcal{F})[M^{2k}] = 1$. Then, the $p(n) \times p(n)$ matrix 
$$
\Big( c_{i_1}\ldots c_{i_r}\left(\mathcal{F}^{j_1}\times \ldots \times \mathcal{F}^{j_s}\right) \Big)
$$
of Chern numbers, where $i_1, \ldots, i_r$ and $j_1, \ldots, j_r$ range over all partitions of $n$, 
has determinant 1. In particular, the columns form a basis over the integers for the Chern numbers.
\end{theorem}
\begin{proof} 
In place of Chern numbers themselves, we may use the basis $s_I$. 
Observe that 
$$
s_I(\mathcal{F}^{j_1}\times \ldots \times \mathcal{F}^{j_s})
= \sum_{I_1, \ldots, I_q = I} 
s_{I_1}(\mathcal{F}^{j_1}\times \ldots \times \mathcal{F}^{j_s}) \cdots 
s_{I_q}(\mathcal{F}^{j_1}\times \ldots \times \mathcal{F}^{j_s})
$$
where the sum is over all partitions $I_1$ of $j_1$, $I_2$ of $j_2$, $\ldots$, and $I_q$ of $j_q$ with 
juxtaposition $I_1\ldots I_q$ equal to $I$. Thus, the characteristic number 
$s_I(E^{j_1}\times \ldots \times E^{j_s})$ is zero unless the partition $I=i_1,\ldots,i_r$ is a 
refinement of $j_1, \ldots, j_q$. In particular, it is zero unless $r \geq q$. Hence, we can 
arrange the partitions $i_1, \ldots, i_q$ and $j_1, \ldots, j_q$ such that the matrix 
$$
\Big( s_{{i_1}\ldots {i_r}}\left(\mathcal{F}^{j_1}\times \ldots \times \mathcal{F}^{j_s}\right) \Big)
$$
is upper triangular with diagonal entries
$$s_{i_1 \ldots i_r}(\mathcal{F}^{i_1}\times \ldots \times \mathcal{F}^{i_r}) 
= s_{i_1}(\mathcal{F}^{i_1}) \cdot s_{i_1}(\mathcal{F}^{i_1}) = 1$$
hence, the matrix has determinant $1$. 
\end{proof}

\begin{proof}[Proof of the Homology Surjection Theorem]
We first reduce to the case of $*=2n$. Let $h:B\Gamma_{n}^\mathbb{C} \to B\Gamma_{n+1}^\mathbb{C}$ be the map which 
classifies the universal $\Gamma_{n}^\mathbb{C}$-structure as a $\Gamma_{n+1}^{\mathbb{C}}$-structure. Then the diagram 
\begin{center}
\begin{tikzcd}
B\Gamma_{n}^\mathbb{C}  
\arrow[d, "Bd"] 
\arrow[r, "h"]& 
B\Gamma_{n+1}^{\mathbb{C}} \arrow[d, "Bd"] \\ 
B\mathrm{GL}(n, \mathbb{C}) \arrow[r, "Bi"] & B\mathrm{GL}(n+1, \mathbb{C})&
\end{tikzcd}
\end{center}
commutes. $Bi_*: H_m(B\mathrm{GL}(n, \mathbb{C}), \mathbb{Z}) \to H_m(B\mathrm{GL}(n + 1, \mathbb{C}), \mathbb{Z})$ is well known to be an isomorphism 
for $m \leq 2n$, thus if $Bd_*: H_{2n}(B\Gamma_n^{\mathbb{C}}, \mathbb{Z}) \to H_{2n}(B\mathrm{GL}(n, \mathbb{C}), \mathbb{Z})$ is surjective then 
so is $Bd_*: H_{2n}(B\Gamma_N^{\mathbb{C}}, \mathbb{Z}) \to H_{2n}(B\mathrm{GL}(N, \mathbb{C}), \mathbb{Z})$ for all $N>n$.

Given a $\Gamma_n^{\mathbb{C}}$-structure $\mathcal{F}$ on a $2n$-dimensional closed oriented manifold $M$, 
it determines Chern numbers by evaluating the Chern classes of $N\mathcal{F}$ on the fundamental class. 
The push forward of the fundamental class $N\mathcal{F}_*([M])$ then determines a unique element in the image of 
$Bd_*: H_*(B\Gamma_n^\mathbb{C}, \mathbb{Z}) \to H_*(B\mathrm{GL}(n, \mathbb{C}), \mathbb{Z})$. Thus, it suffices to construct a 
set of closed oriented manifolds with $\Gamma^{\mathbb{C}}_n$-structures whose Chern numbers form a basis 
over the integers. Now, combining theorem \ref{ExampleTheorem} and theorem \ref{LinearIndependence} yield the conclusion.
\end{proof}

\section{Application to Complex Structures on Open Manifolds} \label{Applications}
This section provides some further applications to the problem of constructing complex analytic coordinates on open manifolds. The first three of these results are slight improvements upon the results of Landweber \cite{Landweber} and Adachi \cite{Adachi2}.

The strategy in all cases is the same; in order to construct complex analytic coordinates, it suffices to 
produce a lift of the map classifying the tangent bundle 
$X \to BGL(2n, \mathbb{R})$ to $B\Gamma_n^\mathbb{C}$. 
The map $B\Gamma_n^\mathbb{C} \to BGL(2n, \mathbb{R})$ factors through $BGL(n, \mathbb{C})$ 
and a lift from $BGL(2n, \mathbb{R})$ to $BGL(n, \mathbb{C})$ corresponds to an almost complex structure 
on the tangent bundle of $X$. So, with an almost complex structure assumed 
we need only to produce a lift of the induced map 
$X \to BGL(n, \mathbb{C})$ to $B\Gamma_n^\mathbb{C}$.

\begin{theorem}
Let $M$ be an open manifold of dimension $2n$ such that 
$$
H^i(M, \mathbb{Z}) = 0
$$ 
for $i > n + 2$. Then $M$ admits a complex structure if and only if it admits an 
almost complex structure. Moreover, any component of almost complex structure contains 
an integrable one.  
\end{theorem}
\begin{proof}
Since the homotopy fiber is $n$ connected then we can always construct a lift over the 
$n+1$ skeleton of $B\mathrm{GL}(n, \mathbb{C})$. Since the primary 
obstruction is natural, corollary \ref{InjectivityCorollary} guarantee that it must vanish. 
Hence it is possible to construct a lift over the $n+2$ skeleton of $B\mathrm{GL}(n, \mathbb{C})$. Thus, it suffices to show that the map classifying $TM$ can be homotopied to lie within the $n+2$ skeleton. The obstructions for doing so live in $H^i(M, \mathbb{Z})$ for $i > n + 2$, hence the vanishing of these groups implies the vanishing of the obstructions.
\end{proof}

\begin{theorem}
For $n > 4$. Let $M$ be an $n$-dimensional smooth manifold whose stable tangent bundle 
admits a complex linear structure. Then $M \times \mathbb{R}^{n-4}$ admits a complex structure. 
\end{theorem}
\begin{proof}
Since the stable tangent bundle of $M$ admits a complex linear structure, then 
$M \times \mathbb{R}^{n-4}$ admits an almost complex structure. Since 
$\mathbb{R}^{n-4}$ is contractible then the map from 
$M \times \mathbb{R}^{n-4}$ to $B\mathrm{GL}(n-2, \mathbb{C}))$ classifying the complex vector bundle 
factors through the $n$-skeleton of $B\mathrm{GL}(n-2, \mathbb{C})$. 
Hence, it admits a lift to $B\Gamma_{n-2}^\mathbb{C}$.
\end{proof}

\begin{theorem}
Let $M$ be an $n$-dimensional orientable smooth manifold with $5 \leq n \leq 7$, and let 
$w_i(M) \in H^i(M, \mathbb{Z}/2\mathbb{Z})$ denote the $i$th Stiefel–Whitney class of $M$. 
If there exist integral cohomology classes $u_2 \in H^2(M, \mathbb{Z})$ and 
$u_6 \in H^6(M, \mathbb{Z})$ such that 
 $u_2 \, \mathrm{ mod } \, 2  = w_2(M)$ and $u_6 \, \mathrm{mod} \, 2 = w_6(M)$. Then 
$M \times \mathbb{R}^{n-4}$ admits a complex structure.
\end{theorem}
\begin{proof}
In light of the previous theorem it suffices to show that the existence 
of integral cohomology classes $u_2 \in H^2(M, \mathbb{Z})$ and 
$u_6 \in H^6(M, \mathbb{Z})$ such that 
$u_2 \, \mathrm{ mod } \, 2  = w_2(M)$ and $u_6 \, \mathrm{mod} \, 2 = w_6(M)$ implies that 
the tangent bundle of $M$ admits a stable complex linear structure. 

The existence of an integral lift of $w_2$ implies that $W_3(M)$ vanishes, which is the primary 
obstruction for an orientable vector bundle to admit a complex structure. 
According to Massey \cite{Massey}, the next obstruction may always be identified with $W_7(M)$ which vanishes 
if and only if $w_6$ admits an integral lift. 
\end{proof}

\begin{theorem} \label{FiniteImpliesConj}
Suppose $\pi_*(B\overline{\Gamma}_n^\mathbb{C})$ is a finite group for $* \leq 2n-2$.
Then conjecture \ref{AlmostComplexImpliesComplex} is true.
\end{theorem}
\begin{proof}
Conjecture \ref{AlmostComplexImpliesComplex} would clearly follow if we could lift the $2n-2$ 
skeleton of $B\mathrm{GL}(n, \mathbb{C})$ to $B\Gamma_n^\mathbb{C}$. 
Let $\mathrm{GL}(n, \mathbb{C})^\delta$ denote $\mathrm{GL}(n, \mathbb{C})$ with the discrete topology, 
then the identity function is a continuous group homeomorphism. 
Moreover, the map factors through $\Gamma_n^\mathbb{C}$ in the obvious way. 

We may assume that $B\mathrm{GL}(n, \mathbb{C})^\delta$ is a subcomplex of $B\mathrm{GL}(n, \mathbb{C})$. 
Since $B\mathrm{GL}(n, \mathbb{C})^\delta$ admits a section, the problem of lifting the $2n-2$ skeleton of 
$B\mathrm{GL}(n, \mathbb{C})$ can be resolved, if one could extend the existing section over the subcomplex 
$B\mathrm{GL}(n, \mathbb{C})^\delta$ to the union of the $2n-2$ skeleton of $B\mathrm{GL}(n, \mathbb{C})$ and 
$B\mathrm{GL}(n, \mathbb{C})^\delta$. 
The obstruction for extending such a section lies in the groups 
$$
H^*(B\mathrm{GL}(n, \mathbb{C}), B\mathrm{GL}(n, \mathbb{C})^\delta, \pi_{*-1}(B\overline{\Gamma}_n^\mathbb{C}))
$$
for $* \leq 2n-2$. Since this is in the stable range, a theorem of 
Suslin \cite{Suslin} implies that these groups are trivial whenever 
$\pi_{*-1}(B\overline{\Gamma}_n^\mathbb{C})$ is a finite group.
\end{proof}
The condition that $\pi_*(B\overline{\Gamma}_n^\mathbb{C})$ is a finite group, 
is much weaker than conjecture \ref{HolomorphicHaefligerThurstonConjecture} and thus the theorem above deserves to be highlighted.
\bibliography{mybib}{}

@article {Adachi1,
    AUTHOR = {Adachi, M.},
     TITLE = {A note on complex structures on open manifolds},
   JOURNAL = {J. Math. Kyoto Univ.},
  FJOURNAL = {Journal of Mathematics of Kyoto University},
    VOLUME = {17},
      YEAR = {1977},
    NUMBER = {1},
     PAGES = {35--46},
      ISSN = {0023-608X},
   MRCLASS = {32C10 (32J99 53C55)},
  MRNUMBER = {435440},
MRREVIEWER = {S.\ Ramanan},
       DOI = {10.1215/kjm/1250522810},
       URL = {https://doi.org/10.1215/kjm/1250522810},
}

@book{MilSta,
  title={Characteristic Classes},
  author={Milnor, John W and Stasheff, James D},
  year={1974},
  publisher={Princeton University Press},
  series={Annals of Mathematics Studies},
  volume={76}
}

@article {Adachi2,
    AUTHOR = {Adachi, M.},
     TITLE = {Construction of complex structures on open manifolds},
   JOURNAL = {Proc. Japan Acad. Ser. A Math. Sci.},
  FJOURNAL = {Japan Academy. Proceedings. Series A. Mathematical Sciences},
    VOLUME = {55},
      YEAR = {1979},
    NUMBER = {6},
     PAGES = {222--224},
      ISSN = {0386-2194},
   MRCLASS = {57R30},
  MRNUMBER = {542397},
MRREVIEWER = {M.\ L.\ Gromov},
       URL = {http://projecteuclid.org/euclid.pja/1195517252},
}

@book {Eliashberg,
    AUTHOR = {Cieliebak, K. and Eliashberg, Y. and Mishachev, N.},
     TITLE = {Introduction to the {$h$}-principle},
    SERIES = {Graduate Studies in Mathematics},
    VOLUME = {239},
   EDITION = {Second},
 PUBLISHER = {American Mathematical Society, Providence, RI},
      YEAR = {[2024] \copyright 2024},
     PAGES = {xvii+363},
      ISBN = {[9781470461058]; [9781470476175]; [9781470476182]},
   MRCLASS = {58Axx (53D05 53D10 57R17)},
  MRNUMBER = {4677522},
}

@incollection {Bott,
    AUTHOR = {Bott, R.},
     TITLE = {On topological obstructions to integrability},
 BOOKTITLE = {Actes du {C}ongr\`es {I}nternational des {M}ath\'ematiciens
              ({N}ice, 1970), {T}ome 1},
     PAGES = {27--36},
 PUBLISHER = {Gauthier-Villars \'Editeur, Paris},
      YEAR = {1971},
   MRCLASS = {57D30},
  MRNUMBER = {425983},
MRREVIEWER = {B.\ Cenkl},
}

@article {Haefligger1,
    AUTHOR = {Haefliger, A.},
     TITLE = {Structures feuillet\'ees et cohomologie \`a{} valeur dans un
              faisceau de groupo\"ides},
   JOURNAL = {Comment. Math. Helv.},
  FJOURNAL = {Commentarii Mathematici Helvetici},
    VOLUME = {32},
      YEAR = {1958},
     PAGES = {248--329},
      ISSN = {0010-2571,1420-8946},
   MRCLASS = {55.00},
  MRNUMBER = {100269},
MRREVIEWER = {R.\ S.\ Palais},
       DOI = {10.1007/BF02564582},
       URL = {https://doi.org/10.1007/BF02564582},
}

@incollection {Haefligger2,
    AUTHOR = {Haefliger, A.},
     TITLE = {Homotopy and integrability},
 BOOKTITLE = {Manifolds--{A}msterdam 1970 ({P}roc. {N}uffic {S}ummer
              {S}chool)},
    SERIES = {Lecture Notes in Math.},
    VOLUME = {Vol. 197},
     PAGES = {133--163},
 PUBLISHER = {Springer, Berlin-New York},
      YEAR = {1971},
   MRCLASS = {57.36},
  MRNUMBER = {285027},
MRREVIEWER = {F.\ Laudenbach},
}

@article {Landweber,
    AUTHOR = {Landweber, P. S.},
     TITLE = {Complex structures on open manifolds},
   JOURNAL = {Topology},
  FJOURNAL = {Topology. An International Journal of Mathematics},
    VOLUME = {13},
      YEAR = {1974},
     PAGES = {69--75},
      ISSN = {0040-9383},
   MRCLASS = {57D30 (57D35)},
  MRNUMBER = {339210},
MRREVIEWER = {M.\ L.\ Gromov},
       DOI = {10.1016/0040-9383(74)90039-1},
       URL = {https://doi.org/10.1016/0040-9383(74)90039-1},
}

@article {Massey,
    AUTHOR = {Massey, W. S.},
     TITLE = {Obstructions to the existence of almost complex structures},
   JOURNAL = {Bull. Amer. Math. Soc.},
  FJOURNAL = {Bulletin of the American Mathematical Society},
    VOLUME = {67},
      YEAR = {1961},
     PAGES = {559--564},
      ISSN = {0002-9904},
   MRCLASS = {57.60},
  MRNUMBER = {133137},
MRREVIEWER = {M.\ A.\ Kervaire},
       DOI = {10.1090/S0002-9904-1961-10690-3},
       URL = {https://doi.org/10.1090/S0002-9904-1961-10690-3},
}

@inproceedings {Suslin,
    AUTHOR = {Suslin, A. A.},
     TITLE = {On the {$K$}-theory of local fields},
 BOOKTITLE = {Proceedings of the {L}uminy conference on algebraic
              {$K$}-theory ({L}uminy, 1983)},
   JOURNAL = {J. Pure Appl. Algebra},
  FJOURNAL = {Journal of Pure and Applied Algebra},
    VOLUME = {34},
      YEAR = {1984},
    NUMBER = {2-3},
     PAGES = {301--318},
      ISSN = {0022-4049,1873-1376},
   MRCLASS = {18F25 (11S70 19D99)},
  MRNUMBER = {772065},
MRREVIEWER = {V.\ P.\ Snaith},
       DOI = {10.1016/0022-4049(84)90043-4},
       URL = {https://doi.org/10.1016/0022-4049(84)90043-4},
}

@article {Nariman,
    AUTHOR = {Nariman, Sam},
     TITLE = {Foliations and diffeomorphism groups},
   JOURNAL = {Notices Amer. Math. Soc.},
  FJOURNAL = {Notices of the American Mathematical Society},
    VOLUME = {71},
      YEAR = {2024},
    NUMBER = {11},
     PAGES = {1471--1483},
      ISSN = {0002-9920,1088-9477},
   MRCLASS = {57R30 (53C12 57S05)},
  MRNUMBER = {4834690},
}

@incollection {Haefliger-Sithanantham,
    AUTHOR = {Haefliger, A. and Sithanantham, K.},
     TITLE = {A proof that {$B\bar \Gamma \sp{{\bf C}}\sb{1}$}\ is
              {$2$}-connected},
 BOOKTITLE = {Symposium on {A}lgebraic {T}opology in honor of {J}os\'e{}
              {A}dem ({O}axtepec, 1981)},
    SERIES = {Contemp. Math.},
    VOLUME = {12},
     PAGES = {129--139},
 PUBLISHER = {Amer. Math. Soc., Providence, RI},
      YEAR = {1982},
      ISBN = {0-8218-5010-5},
   MRCLASS = {57R32 (58H10)},
  MRNUMBER = {676323},
MRREVIEWER = {N.\ V.\ Ivanov},
}

@article {Hurder,
    AUTHOR = {Hurder, S.},
     TITLE = {Independent rigid secondary classes for holomorphic
              foliations},
   JOURNAL = {Invent. Math.},
  FJOURNAL = {Inventiones Mathematicae},
    VOLUME = {66},
      YEAR = {1982},
    NUMBER = {2},
     PAGES = {313--323},
      ISSN = {0020-9910,1432-1297},
   MRCLASS = {57R30 (32L05 32L99)},
  MRNUMBER = {656626},
MRREVIEWER = {Masahisa\ Adachi},
       DOI = {10.1007/BF01389397},
       URL = {https://doi.org/10.1007/BF01389397},
}
\bibliographystyle{plain}
\end{document}